\documentclass[12 pt, leqno, letterpaper]{article}
\usepackage{amssymb,amsmath,amsfonts,amsthm}
\usepackage{enumerate}

\usepackage{pstricks,pst-node}
\usepackage{pdflscape}
\usepackage[none]{hyphenat}
\DeclareMathAlphabet{\chan}{T1}{pzc}{mb}{it}

\newcommand{\mb}{\mathbb}
\newcommand{\mc}{\mathcal}

\newcommand{\bigast}{\text{\large\rm *}}

\newtheorem{thm}{Theorem}

\newenvironment{rem}{\begin{trivlist} \item[] {\bf Remark.}}{\hspace*{0pt}\end{trivlist}}
\newenvironment{pbm}[1][]{\begin{trivlist} \item[] {\bf Problem#1.}}{\hspace*{0pt}\end{trivlist}}

\newsavebox{\Theoremtype}
\newsavebox{\Theoremlabel}

\newtheoremstyle{break}
{\topsep}
{\topsep}
{\it}
{}
{}
{}
{ }
{\thmname{\textbf{#1}}\thmnumber{ \textbf{#2.}}}

\theoremstyle{break}

\newtheoremstyle{ref}
{\topsep}	
{\topsep}	
{\it}
{}
{}
{}
{ }
{\thmname{{\bfseries#1}}\thmnumber{ \bfseries#2\thmnote{\rm #3}\bfseries .}}

\theoremstyle{ref}
\newtheorem{lem}[thm]{Lemma}

\newtheorem{cor}[thm]{Corollary}

\newtheoremstyle{nnref}
{\topsep}	
{\topsep}	
{\rm}
{}
{}
{}
{ }
{\thmname{\textbf{#1}\thmnote{\textrm{ #3}}\textbf{.}}}

\theoremstyle{nnref}
\newtheorem{defn}{Definition}

\newcommand{\Kechrisref}{\cite[\!I.7.8]{Kechris1994}}
\newcommand{\Engelkingref}{\cite[4.4J]{Engelking1989}}

\newcommand{\Engelkingrefi}{\cite[3.6.6]{Engelking1989}}
\newcommand{\Kechrisi}{\cite[7.10]{Kechris1994}}
\newcommand{\set}[1]{\left\{#1\right\}}

\begin{document}
\sloppy
\title{Lindel\"of spaces which are indestructible, productive, or $D$}
\author{Leandro F. Aurichi\makebox[0cm][l]{$^1$}~~and Franklin D. Tall\makebox[0cm][l]{$^2$}}

\footnotetext[1]{Research supported by Conselho Nacional de Desenvolvimento Cient\'ifico e Tecnol\'ogico and Coordena\c c\~ao de Aperfei\c coamento de Pessoal de N\'ivel Superior (Brazil). This research was undertaken while the first author was visiting the Department of Mathematics of the University of Toronto. He thanks the Department and especially the members of the Set Theory Seminar for their hospitality.}
\footnotetext[2]{Research supported by grant A-7354 of the Natural Sciences and Engineering Research Council of Canada.\vspace*{1pt}}
\date{\today}
\maketitle

\begin{abstract}
We discuss relationships in Lindel\"of spaces among the properties ``indestructible'', ``productive'', ``$D$'', and related properties.
\end{abstract}

\renewcommand{\thefootnote}{}
\footnote
{\parbox[1.8em]{\linewidth}{$(2010)$ Mathematics Subject Classification. Primary 54D20, 54D99, 03E35; Secondary 54G20.}\vspace*{3pt}}
\renewcommand{\thefootnote}{}
\footnote
{\parbox[1.8em]{\linewidth}{Key words and phrases: Lindel\"of, indestructible, countably closed forcing, productively Lindel\"of, $D$, Alster, Menger, Hurewicz.}}

\section{Introduction}

The question of what additional assumptions ensure that the product of two Lindel\"of spaces is Lindel\"of is natural and well-studied. See e.g., \cite{Michael1963}, \cite{Michael1971}, \cite{Alster1987}, \cite{Alster1988}, \cite{Przymusinski1980}, \cite{Przymusinski1984}. The question of which topological properties are preserved by which kinds of forcing is also a natural one.  See e.g., \cite{Dow1990}, \cite{Tall1995}, \cite{Grunberg1998}, \cite{Iwasa2007}, \cite{TallProblems}, etc.  The question of whether a Lindel\"of space remains Lindel\"of after countably closed forcing is particularly interesting because of its connection with the classic problem of whether Lindel\"of spaces with points $G_\delta$ consistently have cardinality $\leq2^{\aleph_0}$ \cite{Tall1995}.  We need some definitions.

\begin{defn}
A space is \textit{indestructibly Lindel\"of} if it is Lindel\"of in every countably closed forcing extension.

Note that indestructibly Lindel\"of spaces are Lindel\"of.
\end{defn}

$D$-spaces were introduced in \cite{vanDouwen1979}.

\begin{defn}
A space $X$  is $D$ if for every \textit{neighbourhood assignment} $\{V_x\}_{x\in X}$, i.e. each $V_x$ is an open set containing $x$, there is a closed discrete $Y\subseteq X$ such that $\{V_x\}_{x\in Y}$ covers $X$.  $Y$ is called a \textit{kernel} of the neighbourhood assignment.
\end{defn}

The question raised in \cite{vanDouwen1979} of whether every Lindel\"of space is a $D$-space has been surveyed in~\cite{Eisworth2007} and \cite{Gruenhage2009}. It has recently been the subject of much research.  In \cite{Aurichi}, the first author established many connections between the $D$ property, topological games, and selection properties.  In this paper, we examine Lindel\"of indestructibility and connections of it, selection principles, and the $D$ property with preservation of Lindel\"ofness under products.

\begin{defn}[\cite{Barr2000}]
A space $X$ is \textit{productively Lindel\"of} if $X \times Y$ is Lindel\"of for any Lindel\"of space $Y$.
\end{defn}

\begin{defn}
A space is \textit{indestructibly productively Lindel\"of} if it is productively Lindel\"of in any extension by countably closed forcing.
\end{defn}

\begin{defn}[\cite{Alster1988}, \cite{Barr2000}]
A space is \textit{Alster} if every cover by $G_\delta$ sets that covers each compact set finitely includes a countable subcover.
\end{defn}

\begin{defn}
A space is \textit{indestructibly $D$} if it remains $D$ after countably closed forcing.
\end{defn}

We shall deal with three selection principles in this paper.  These principles have a variety of equivalent definitions --- see e.g.~\cite{Scheepers}.  Here are the first two.  The third -- \textit{Hurewicz} -- is defined in Section 4.

\begin{defn}
A space is \textit{Menger} if for each sequence $\{\mc{U}_n\}_{n<\omega}$ of open covers, such that each finite union of elements of $\mc{U}_n$ is a member of $\mc{U}_n$, there are $U_n\in\mc{U}_n$, $n < \omega$, such that $\{U_n : n<\omega\}$ is an open cover.  A space is \textit{Rothberger} if for each sequence $\{\mc{U}_n\}_{n<\omega}$ of open covers, there are $U_n \in \mc{U}_n$, such that $\{U_n : n<\omega\}$ is an open cover.
\end{defn}

Clearly every Rothberger space is Menger.
Previously known results include:
\begin{lem}[~\cite{Scheepers}]
Every Rothberger space is indestructibly Lindel\"of.
\end{lem}

\begin{lem}[~\cite{Alster1988}]\label{lem1}
Every Alster space is productively Lindel\"of; $CH$ implies every productively Lindel\"of $T_3$ space of weight $\leq\aleph_1$ is Alster. Alster metrizable
spaces are $\sigma$-compact.
\end{lem}

\section{$D$-spaces}

We shall now see what we can say about $D$-spaces. Theorem \ref{thm3} and Corollary \ref{cor4} are the important results in this section. No result resembling Corollary \ref{cor5} was previously known.

\begin{lem}[~\cite{Aurichi}]\label{lem2}
Every Menger space is $D$.
\end{lem}

\begin{thm}\label{thm3}
Every Alster space is Menger.
\end{thm}

\begin{proof}
Let $\{\mc{U}_n\}_{n<\omega}$ be a sequence of open covers of $X$, each closed under finite unions.  Let $\mc{G}$ be the set of all $\bigcap_{n<\omega} U_n\text{'s}$, where $U_n \in \mc{U}_n$.  Let $K$ be any compact subspace of $X$.  Then for each $n<\omega$, $K$ is included in some $U_n\in\mc{U}_n$.  Thus $K$ is included in some $G\in\mc{G}$. Since $X$ is Alster, there are $\{H_k\}_{k<\omega}$ in $\mc{G}$ such that $\bigcup_{k<\omega}H_k$ covers $X$.  Let $H_k = \bigcap_{n<\omega} U_{nk}$, where $U_{nk} \in \mc{U}_n$.  Then $\{U_{nn}\}_{n<\omega}$ covers $X$, since $H_n \subseteq U_{nn}$.  Thus, since each $U_{nn} \in \mc{U}_n$, $X$ is Menger.
\end{proof}

\begin{cor}\label{cor4}
Every Alster space is $D$.
\end{cor}

\begin{cor}\label{cor5}
$CH$ implies every productively Lindel\"of $\,T_3$ space which is either first countable or separable is $D$.
\end{cor}

\begin{proof}
First countable Lindel\"of Hausdorff spaces have cardinality and hence weight $\leq 2^{\aleph_0}$; separable regular spaces have weight $\leq 2^{\aleph_0}$.
\end{proof}

\section{Indestructibly productively Lindel\"of\\ spaces}

``Indestructibly productively Lindel\"of'' is much harder to understand than
is ``indestructibly Lindel\"of''. In a previous version of this note, we
prematurely claimed that indestructibly productively Lindel\"of $T_3$ spaces
are Alster. This may well be true, but, at the moment, we do not have a
proof. We can, however, prove this for metrizable spaces, which is the key
result of this section.

\begin{thm}\label{thm7}
A metrizable space is indestructibly productively Lindel\"of if and only
if it is $\sigma$-compact.
\end{thm}

\begin{proof}
The backward direction is routine, since ``$\sigma$-countably-compact'' and ``metrizable'' are both preserved by countably closed forcing, and so hence is ``$\sigma$-compact metrizable''. It is well-known that $\sigma$-compact spaces are productively Lindel\"of. For the other direction, first recall:

\begin{lem}[~(\Engelkingref)]\label{lem18}
Every separable metrizable space is a perfect image of a $0$-dimensional separable metrizable space.
\end{lem}

We claim that if we prove Theorem \ref{thm7} for $0$-dimensional spaces,
it will follow for all spaces. For suppose $X$ is an indestructibly
productively Lindel\"of metrizable space and $X'$ is its $0$-dimensional
perfect pre-image by a map $f$. Note that a space with a countable base
has no new open or closed sets in a countably closed extension. Thus $f$
remains continuous and closed in the extension. It may not be perfect, but
inverse images of points are Lindel\"of. Let $Y$ be Lindel\"of in the
extension. Then $f \times \mathrm{id}_Y$ also is continuous, closed, and
has inverse images of points Lindel\"of. Since $X \times Y$ is Lindel\"of,
it follows that $X' \times Y$ is Lindel\"of. Thus $X'$ is indestructibly
productively Lindel\"of. It is then $\sigma$-compact and therefore so is $X$.
Thus without loss of generality, we shall assume our space $X$ is
$0$-dimensional, and hence can be considered as a subspace of the Cantor set
$\mathbb{K}$.

Collapse $2^{\aleph_0}$ to $\aleph_1$ by countably closed forcing. Then $X$ remains productively
Lindel\"of. By Lemma \ref{lem1}, it is $\sigma$-compact in the extension.
As noted above, $\mathbb{K}$ has no new closed sets and hence no new
$F_\sigma$'s. Thus if $X = \bigcup_{n<\omega} F_n$ in the extension, where
the $F_n$'s are compact and hence closed subspaces of $\mathbb{K}$, then the
$F_n$'s are actually in the ground model and hence compact there as well.
\end{proof}

There are some other conditions that imply indestructibly productively
Lindel\"of:

%
%

\begin{thm}\label{thm6}
Every Lindel\"of space which either is scattered or is a $P$-space
or is indestructibly Lindel\"of and $\sigma$-compact
is indestructibly productively Lindel\"of.
\end{thm}

\begin{proof}
This follows easily since:
\begin{lem}[~\cite{Juhasz1989}]
The Lindel\"ofness (and scatteredness) of a scattered space is preserved by any forcing.
\end{lem}
\begin{lem}[~\cite{Barr2000}]\label{lem7}
Every Lindel\"of space which either is scattered or is a $P$-space is Alster.
\end{lem}
\begin{lem}[~\cite{Scheepers}]
Lindel\"of $P$-spaces are Rothberger and hence indestructible.
\end{lem}

Clearly the $P$-property ($G_\delta$'s open) is preserved by countably closed forcing. That indestructibly Lindel\"of $\sigma$-compact spaces are indestructibly productively Lindel\"of follows from the proof of Theorem \ref{thm7}.
\end{proof}

We don't know whether indestructibly Lindel\"of or productively Lindel\"of implies $D$. Indestructibly productively Lindel\"of does \cite{TallNote}. Note
that an indestructibly Lindel\"of non-$D$ space remains non-$D$ in any countably closed forcing extension.

\begin{thm}
Suppose $X$ is Lindel\"of, $D$, and countably tight.  Then $X$ is indestructibly Lindel\"of if it is indestructibly $D$.
\end{thm}
\begin{proof}
In \cite{Tall1995} it is shown that after countably closed forcing, Lindel\"of countably tight spaces retain countable extent.  But $D$-spaces with countable extent are Lindel\"of.
\end{proof}

\begin{thm}
Suppose $X$ is Lindel\"of, $D$, $|X| \leq \aleph_1$. Then $X$ is both indestructibly Lindel\"of and indestructibly $D$.
\end{thm}
\begin{proof}
That Lindel\"of spaces of size $\leq \aleph_1$ are indestructible was proved in \cite{Tall1995}. Suppose $\{\dot{V}_{x_\alpha}\}_{\alpha<\omega_1}$ is a neighbourhood assignment in the extension; without loss of generality, we may assume each $\dot{V}_{x_\alpha}$ is a ground model open set. Given an arbitrary condition $p$ forcing all this, take below $p$ a descending sequence of conditions $\{p_\alpha\}_{\alpha<\omega_1}$ deciding $\dot{V}_{x_\alpha}$. The resulting $V_{x_\alpha}$'s form a neighbourhood assignment in the ground model. It had a countable kernel $\{x_{\alpha_n}\}_{n<\omega}$. Let $\alpha_\omega \ge \text{ each } \alpha_n$. Then $p_{\alpha_\omega}$ forces $(\{x_{\alpha_n}\}_{n<\omega}){\check{\hspace*{1.5pt}}}$ is a kernel for $\{\dot{V}_{x_\alpha}\}_{\alpha<\omega_1}$.
\end{proof}

\begin{cor}\label{cor13}
$CH$ implies productively Lindel\"of first countable $T_3$ spaces are indestructibly $D$.
\end{cor}

Productively Lindel\"of $D$-spaces are not necessarily indestructibly $D$: consider the usual product topology on $2^{\omega_1}$.  Adding a Cohen subset of $\omega_1$ with countable conditions makes $2^{\omega_1}$ non-Lindel\"of~\cite{Tall1995}, but countably closed forcing preserves countable compactness. Since countably compact plus $D$ = compact, we see that the forcing does not preserve $D$.  Thus countably closed forcing does not preserve Menger or Alster.  Neither does it preserve productively Lindel\"of.  To see this, again consider adding a Cohen subset of $\omega_1$ with countable conditions.  $2^{\omega_1}$ is productively Lindel\"of in the ground model; its weight is $\aleph_1$, so in the extension, if it were productively Lindel\"of, it would be Alster, since $CH$ holds.

Since Rothberger implies both indestructibly Lindel\"{o}f and Menger, one might wonder if it is strong enough to imply productively Lindel\"of.  It is not; see Section 9 below.  Thus Lindel\"of productivity is not a necessary condition for $D$-ness in Lindel\"of spaces.  Alster does not imply Rothberger, since Alster is equivalent to $\sigma$-compact in metrizable spaces~\cite{Alster1988}, but Rothberger subsets of the real line have strong measure zero --- see e.g.~\cite{Miller1984}, where they are called $C''$ sets.  Similarly, indestructibly productively Lindel\"of does not imply Rothberger --- consider the closed unit interval.

Indestructibly Lindel\"of spaces need not be productively Lindel\"of; a Bernstein (totally imperfect) set of reals provides a counterexample~\cite{Michael1971}.  Alster does not imply indestructibly productively Lindel\"of, since $\sigma$-compact spaces are Alster~\cite{Alster1988}.

Among the properties we have considered so far, the interesting open questions (say for $T_3$ spaces) are:
{\it
\begin{enumerate}
\item Do any of Lindel\"of, indestructibly Lindel\"of, productively Lindel\"of imply $D$?
\item Does productively Lindel\"of imply Alster \cite{Barr2000}? (This question was first asked in \cite{Alster1988}, with different terminology.) Does
indestructibly productively Lindel\"of imply Alster?
\item Are indestructibly Lindel\"of, productively Lindel\"of spaces indestructibly productively Lindel\"of?
\item Are indestructibly Lindel\"of $D$-spaces indestructibly $D$?
\end{enumerate}
}

%

\section{Productively Lindel\"of completely\\ metrizable spaces}

The question of whether productively Lindel\"of spaces are Alster reduces in the metrizable case to  whether they are $\sigma$-compact. The second author examines this in detail in \cite{Tall}; here we shall mainly confine ourselves to the completely metrizable case. Recall the famous problem of E.~Michael which asks whether there is a Lindel\"of space whose product with the space $\mathbb{P}$ of irrationals is not Lindel\"of. See~\cite{Michael1963}, \cite{Michael1971}, \cite{Moore1999}. We shall prove:

\begin{thm}\label{thm14}
The following assertions are equivalent:
\begin{enumerate}[a)]
\item Every completely metrizable productively Lindel\"of space is Menger,
\item Every completely metrizable productively Lindel\"of space is Alster,
\item Every completely metrizable productively Lindel\"of space is $\sigma$-compact,
\item There is a Lindel\"of space $X$ such that $X \times \mathbb{P}$ is not Lindel\"of.
\end{enumerate}
\end{thm}
\begin{proof}
As mentioned earlier, a metrizable space is Alster if and only if it is $\sigma$-compact, if and only if it is indestructibly productively Lindel\"of. To show a), b), c) equivalent, then, it suffices to prove a) implies c). We note that every productively Lindel\"of space is Lindel\"of, and hence, if metrizable, is separable.

We next need:

\begin{lem}[~\Kechrisref]\label{lem19}
Every $0$-dimensional separable metrizable \v{C}ech- complete space is homeomorphic to a closed subspace of $\mathbb{P}$, the space of irrationals, considered as $\omega^\omega$.
\end{lem}

These yield:

\begin{lem}\label{lem20}
If there is a productively Lindel\"of completely metrizable space which is not $\sigma$-compact, then there is one included in $\mathbb{P}$.
\end{lem}

\begin{proof}
As in the proof of Theorem \ref{thm7}, if $X'$ maps perfectly onto a productively Lindel\"of $X$, then $X'$ is productively Lindel\"of. Next, recall that a perfect image of a $\sigma$-compact space is $\sigma$-compact. If then $X$ is not $\sigma$-compact, then neither is $X'$. Finally, the perfect pre-image of a completely metrizable space is completely metrizable.
\end{proof}


Hurewicz \cite{Hurewicz1925} proved that analytic (and, in particular, $G_\delta$) sets of reals are Menger if and only if they are $\sigma$-compact.  Thus, we see that, assuming a),  $0$-dimensional completely metrizable productively Lindel\"{o}f spaces are $\sigma$-compact. The non-$0$-dimensional case then follows.


Having established that a) implies c) we next prove that c) is equivalent to d). Since $\mathbb{P}$ is not $\sigma$-compact, it follows that if every completely metrizable productively Lindel\"of space is $\sigma$-compact, then there must be a Lindel\"of space $X$ with $X \times \mathbb{P}$ not Lindel\"of. Conversely, assume there is a completely metrizable productively Lindel\"of space $Y$ which is not $\sigma$-compact. Let $X$ be any Lindel\"of space. Claim: $X \times \mathbb{P}$ is Lindel\"of.  Recall Hurewicz's Theorem:

\begin{lem}[(\Kechrisi)]
If $Y$ is a completely metrizable Lindel\"{o}f space which is not $\sigma$-compact, then $Y$ includes a closed copy of $\mathbb{P}$.
\end{lem}

It follows that $X \times \mathbb{P}$ is a closed subspace of $X \times Y$.  Since $X \times Y$ is Lindel\"{o}f, so is $X \times \mathbb{P}$.\hfill\qed\newline\newline
Cardinal invariants of the continuum are closely related to Michael's problem.

\begin{defn}
Partially order $^\omega \omega$ by $f \leq\!\!\bigast\ g$ if $f(n) \leq g(n)$ for all but finitely many $n$. $\mathfrak{b}$ is the least cardinal such that for every $\mc{F} \subseteq {}^\omega\omega$ of size $< \mathfrak{b}$, there is a $g \in~^\omega\omega$ such that for each $f \in \mc{F}$, $f \leq\!\!\bigast\ g$.  $\mathfrak{d}$ is the least cardinal $\delta$ such that there is a family $\mc{F}$ of size $\delta$ included in $^\omega \omega$ such that for every $f \in\ ^\omega \omega$, there is a $g \in \mc{F}$, such that $f \leq\!\!\bigast\ g$. $\mathop{cov}(\mc{M})$ is the least cardinal $\delta$ such that $\ ^\omega \omega$, identified with $\mathbb{P}$, is the union of $\delta$ nowhere dense sets. A
  \textbf{$\lambda$-scale} is a subset $S$ of $^\omega\omega$ of size
  $\lambda$ such
  that $<\bigast$ (i.e., $\leq\bigast$, but not for all but finitely
  many $n$ equal) well-orders $S$ and each $f \in\ ^\omega\omega$ is
  less than some member of $S$.
\end{defn}

\begin{lem}[~\cite{vanDouwen1984}]
$\mathfrak{b} = \aleph_1$ implies there is a Lindel\"of regular space $X$ such that $X \times \mathbb{P}$ is not Lindel\"of.
\end{lem}

\begin{lem}[~\cite{Moore1999}]
$\mathfrak{d} = \mathop{cov}(\mc{M})$ implies there is a Lindel\"of regular space $X$ such that $X \times \mathbb{P}$ is not Lindel\"of.
\end{lem}

\begin{cor}\label{cor22}
$\mathfrak{b} = \aleph_1$ or $\mathfrak{d} = \mathop{cov}(\mathcal{M})$ implies every productively Lindel\"of, completely metrizable space is $\sigma$-compact.
\end{cor}

It is interesting to wonder whether there is a test space for whether productively Lindel\"of metrizable spaces are $\sigma$-compact, as is provided by $\mb{P}$ in the completely metrizable case. Under $CH$, by Lemma \ref{lem1} every productively Lindel\"of metrizable space is $\sigma$-compact. This is proved explicitly in \cite{Alster1987}. If there were a productively Lindel\"of, metrizable, non-$\sigma$-compact space, it would be an example of a productively Lindel\"of non-Alster space, and of a productively Lindel\"of, indestructibly Lindel\"of space which would not be indestructibly productively Lindel\"of.

In a previous version of this note, we claimed that $\mathfrak{d} = \aleph_1$ implies productively Lindel\"of metrizable spaces are $\sigma$-compact, improving the $CH$ result referred to above. The referee pointed out an error in our proof, which we have been unable to fix. We are no longer confident of the truth of our claim.

\begin{pbm}{}
Does $\mathfrak{d} = \aleph_1$ imply productively Lindel\"of metrizable spaces are $\sigma$-compact?
\end{pbm}

As a consolation prize, we shall prove a weaker assertion.

\begin{defn}
A \textbf{$\gamma$-cover} of a space is a countably infinite open cover such that each point is in all but finitely many members of the cover.  A space is \textbf{Hurewicz} if, given a sequence $\set{\mc{U}_n : n \in \omega}$ of $\gamma$-covers, there is for each $n$ a finite $\mc{V}_n \subseteq \mc{U}_n$ such that either $\set{\bigcup\mc{V}_n : n \in \omega}$ is a $\gamma$-cover, or else for some $n$, $\mc{V}_n$ is a cover.
\end{defn}

\begin{thm}
$\mathfrak{d} = \aleph_1$ implies every productively Lindel\"{o}f metrizable space is Hurewicz.
\end{thm}

The Hurewicz property fits strictly between Menger and $\sigma$-compact.  See e.g. \cite{Just1996}, \cite{Tsaban2008}, and \cite{Ts}.  In a successor \cite{Tall} to this paper, the second author proves that Alster implies Hurewicz.  There are a number of equivalent definitions of Hurewicz -- see \cite{Just1996}, \cite{Tsaban2008}, \cite{BZ}, \cite{Tall}.

It will be convenient to work with one of them.  To avoid relying on the unpublished \cite{Tall}, we shall temporarily call this property \textit{Hurewicz}$^*$:

\begin{defn}
A Lindel\"{o}f $T_3$ space is \textbf{Hurewicz$^*$} if and only if every \v{C}ech-complete $Y \supseteq X$ includes a $\sigma$-compact $Z \supseteq X$.
\end{defn}

Banakh and Zdomskyy \cite{BZ} prove that:

\begin{lem}\label{newlem25}
Hurewicz$^*$ is equivalent to Hurewicz in separable metrizable spaces.
\end{lem}

We generalized this to Lindel\"{o}f $T_3$ spaces in \cite{Tall}, but their version is all we need here.  We next observe:

\begin{lem}\label{newlem26}
A $T_{3\frac{1}{2}}$ perfect image of a Hurewicz$^*$ $T_{3\frac{1}{2}}$ space is Hurewicz$^*$.
\end{lem}

\begin{proof}
Let $p : X$ onto $X_0$ be perfect.  Let $Y_0$ be a \v{C}ech-complete space including $X_0$.  Then the closure $\overline{X_0}$ of $X_0$ in $Y_0$ is also \v{C}ech-complete.  Then $\beta\overline{X_0}$ is a compactification of $X_0$.  Recall:

\begin{lem}[(\Engelkingrefi)]
For every compactification $\alpha T$ of a $T_{3\frac{1}{2}}$ space $T$ and every continuous map $f : S \to T$ of a $T_{3\frac{1}{2}}$ space $S$ to the space $T$, there is a continuous extension $F: \beta S \to \alpha T$ over $\beta S$ and $\alpha T$.
\end{lem}

Thus we may extend $p : X \to X_0$ to $P : \beta X \to \beta \overline{X_0}$.  Let $Y = P^{-1}(\overline{X_0})$.  Then $Y$ is a \v{C}ech-complete space including $X$, since \v{C}ech-completeness is a perfect invariant for $T_{3\frac{1}{2}}$ spaces \cite{Engelking1989}.  Let $W$ be $\sigma$-compact, $X \subseteq W \subseteq Y$.  Then $P(W)$ is $\sigma$-compact, $X_0 \subseteq P(W) \subseteq Y_0$.
\end{proof}

By Lemmas \ref{lem18}, \ref{newlem25}, and \ref{newlem26}, and since it is easy to see that productive Lindel\"{o}fness is a perfect invariant, we may conclude that if there is a productively Lindel\"{o}f space which is not Hurewicz, there is one included in $\mathbb{P}$.  Furthermore, by the following result of Rec\l{}aw \cite{R}, we may further assume that there is a productively Lindel\"{o}f $X \subseteq \mathbb{P}$ such that $X$ is not included in any $\sigma$-compact subspace of $\mathbb{P}$.
\begin{lem}[~\cite{R}]\label{lem26}
A $0$-dimensional subset $X$ of $\mb{P}$ is Hurewicz if and only if every homeomorph of $X$ included in $\mb{P}$ is included in a $\sigma$-compact subspace of $\mb{P}$.
\end{lem}

Let $\{ f_\alpha : \alpha < \omega_1 \}$ be a dominating family for $^\omega\omega$, thinned out to form a scale. For each $\alpha < \omega_1$, let $x_\alpha \in X$ be such that $x_\alpha \not \le^\ast f_\beta$, for every $\beta < \alpha$. There always is such an $x_\alpha$, else $X$ would be included in a $\sigma$-compact subspace of $\mathbb{P}$. Considering $\mathbb{P}$ as a subspace of $[0,1]$, let $Y' = [0,1] - X$. Let $Y = Y' \cup \{x_\alpha : \alpha < \omega_1 \}$. Strengthen the topology on $Y$ by making all the $x_\alpha$'s isolated. Then claim $Y$ is still Lindel\"of. For if $V \supseteq Y'$ is open in $[0,1]$, then $[0,1] - V$ is compact in $[0,1]$ and included in $X$.
Then some $f_\alpha$ bounds it. Since the $f_\alpha$'s form a scale, none
of the $x_\beta$'s, for $\beta \ge \alpha$ are $\le^\ast f_\alpha$.
Therefore there are
only countably many $x_\alpha$'s in $[0,1] - V$. Then any open cover of $Y$
will include countably many open sets which cover all but countably many
members of $Y$. The usual argument shows that $X \times Y$ is not Lindel\"of,
since $\{\langle x_\alpha, x_\alpha \rangle: \alpha < \omega_1\}$ is uncountable closed discrete.
\end{proof}

\section{Other productive properties}

There are some other properties we may productively consider.

\begin{defn}
A space is \textit{powerfully Lindel\"of} if its $\omega$th power is Lindel\"of. A space is \textit{finitely powerfully Lindel\"of} if all of its finite powers are Lindel\"of.  (Finitely powerfully Lindel\"of spaces are called \textit{$\varepsilon$-spaces} in \cite{Gerlits1982}.)
\end{defn}

\begin{lem}[~\cite{Alster1988}]
Alster spaces are powerfully Lindel\"of.
\end{lem}

\begin{lem}[~\cite{Alster1988}]
Productively Lindelof spaces are finitely powerfully Lindelof.
\end{lem}

Przymusi\'nski \cite{Przymusinski1980} has constructed a finitely powerfully Lindel\"of space that is not powerfully Lindel\"of. Michael \cite{Michael1971} constructed a subset $M$ of the real line and a Lindel\"of space such that the product of the two was not Lindel\"of. Thus $M$ is powerfully Lindel\"of but not productively Lindel\"of. He also proved:

\begin{lem}[~\cite{Michael1971}]
If $X^{\omega}$ is normal, then $X \times \mathbb{P}$ is normal.
\end{lem}

\noindent On the other hand,
\begin{lem}[~\cite{Rudin1975}]
Suppose $X$ is Lindel\"{o}f regular and $Y$ is separable metrizable.  Then $X \times Y$ is normal if and only if $X \times Y$ is Lindel\"{o}f.
\end{lem}
It follows that:
\begin{thm}
If $X$ is regular and powerfully Lindel\"of, then $X \times \mathbb{P}$ is Lindel\"of.
\end{thm}

Michael also raised the following question (the earliest reference we have found is \cite{Przymusinski1984}), which is still unsolved:

{\it
\begin{enumerate}
\item[5.] Are productively Lindel\"of spaces powerfully Lindel\"of?
\end{enumerate}
}

We can give a partial answer:

\begin{defn}
A space is \textit{productively $FC$-Lindel\"of} if its product with every first countable Lindel\"of $T_3$ space is Lindel\"of.
\end{defn}

\begin{thm}
$CH$ implies if $X$ is first countable $T_3$, and productively $FC$-Lindel\"of, then $X$ is Alster, and hence $X$ is powerfully Lindel\"of.
\end{thm}
\begin{proof}
Note that $X$ is Lindel\"of and hence has weight $\leq 2^{\aleph_0}$.  Assuming $CH$, given a non-Alster Lindel\"of $T_3$ space $X$ of weight $\leq \aleph_1$, Alster \cite{Alster1988} constructs a space $Y' = P \cup A$ such that:
\begin{enumerate}
\item $Y'$ is Lindel\"of,
\item $X \times Y'$ is not Lindel\"of,
\item $P$ is a set of isolated points,
\item $Y'$ is a subspace of a space $Y$ in which each $a \in A$ has a countable neighbourhood base.
\end{enumerate}

But then $Y'$ is first countable.
\end{proof}

An example of a finitely powerfully Lindel\"of space whose product with $\mathbb{P}$ is not Lindel\"of can be constructed by Michael's original construction. Recall that construction produces from $CH$ an uncountable subset $C$ of $\mathbb{R}$ concentrated on the rationals. $C$, as a subspace of the Michael line, is Lindel\"of, yet $C \times \mathbb{P}$ is not. We simply need such a $C$ with $C^n$ Lindel\"of for every $n$. Michael in fact constructs such a $C$ and hence such a space from $CH$ in \cite{Michael1971}.  In an earlier version of this note, we claimed we could get this from $\mathfrak{b} = \aleph_1$, but B. Tsaban found an error in the proof, so this remains open.

\section{The Rothberger property and\\ concentrated sets}

There are some more points concerning the Rothberger property worth noting.

\begin{defn}
We say that a topological space is \textit{concentrated} on $Y \subseteq X$ if, for every open set $U$ such that $U \supseteq Y$, $X \smallsetminus U$ is countable.
\end{defn}

\begin{thm}[folklore]
If $X$ is concentrated on a Rothberger (Menger) subspace, then $X$ is Rothberger (Menger).
\end{thm}

\begin{proof}
We will prove the Rothberger case; the Menger case is analogous. Note that $X$ is Lindel\"of. Let $(\mc{U}_n)_{n \in \omega}$ be a sequence of open coverings for $X$. Let $Y$ be a Rothberger subspace such that $X$ is concentrated on it. Let $(U_{2n})_{n \in \omega}$ be a covering for $Y$ such that each $U_{2n} \in \mc{U}_{2n}$. Let $\{x_n: n \in \omega\} = X \smallsetminus \bigcup_{n \in \omega} U_{2n}$. For each $n \in \omega$, pick $U_{2n + 1} \in \mc{U}_{2n + 1}$ such that $x_n \in U_{2n + 1}$. Note that $(U_n)_{n \in \omega}$ is a covering for $X$.
\end{proof}

\begin{defn}
A space is \textit{Lusin} if every nowhere dense set is countable.
\end{defn}

\begin{cor}
Every separable Lusin space is Rothberger and, therefore, $D$.
\end{cor}

\begin{proof}
Observe that a separable Lusin space is concentrated on a countable set.
\end{proof}

Separability cannot be dispensed with. A \textit{Sierpi\'nski set} is an uncountable set of reals which has countable intersection with every null set. Sierpi\'nski sets exist under $CH$; they are Lusin in the density topology on the real line, indeed the null sets coincide with the first category sets -- see \cite{Tall1976} for details. Rothberger sets have (strong) measure 0 \cite{Miller1984},  so Sierpi\'nski sets cannot be Rothberger in the usual topology on the real line and hence not in any strengthening of that.

Michael's space is concentrated on the rationals, so is Rothberger. Thus it is consistent that a Rothberger space (therefore a $D$-space) need not be productively Lindel\"of even if the products are taken only with well-behaved spaces such as $\mathbb{P}$.

\section{Elementary submodels}

A similar argument to the forcing one following Corollary \ref{cor13} shows that elementary submodels do not preserve $D$ or Lindel\"of.  Let $M$ be a countably closed elementary submodel of some $H_\theta$,  where $\theta$ is a sufficiently large regular cardinal, with the compact space $X$ in $M$.  Then the space $X_M$ defined in \cite{Junqueira1998}, namely the topology on $X \cap M$ generated by $\{U\cap M: U \in M \text{ and $U$ open in } X\}$ is countably compact~\cite{Junqueira1998}.  On the other hand, if we take $X$ to be e.g. $2^{2^{\aleph_0}}$ and $|M|=2^{\aleph_0}$, $X_M$ will not be compact \cite{Junqueira1998}, and hence not $D$ nor Lindel\"of.

However, we have:

\begin{thm}
If $X$ is a first countable $T_2$ $D$-space, and $M$ is a countably closed elementary submodel containing $X$, then $X_M$ is $D$.
\end{thm}
\begin{proof}
In such a situation, $X_M$ is a closed subspace of $X$~\cite{Junqueira1998}.  Closed subspaces of $D$-spaces are easily seen to be $D$-spaces.
\end{proof}

\begin{thm}
If $X$ is $T_2$ and of pointwise countable type and hereditarily $D$, then $X_M$ is $D$.
\end{thm}
\begin{proof}
By \cite{Junqueira1998}, if $X$ is $T_2$ and of pointwise countable type, then $X_M$ is a perfect image of a subspace of $X$.  By \cite{Borges1991}, perfect images of $D$-spaces are $D$.
\end{proof}

Alster \cite{Alster1988} asked whether if $CH$ holds and every closed subspace of $X$ of weight $\leq \aleph_1$ is Alster, then $X$ must be Alster.  We can prove this for first countable spaces:

\begin{thm}
Suppose $X$ is first countable $T_2$ and each closed subspace of $X$ of weight $\leq 2^{\aleph_0}$ is Alster. Then $X$ is Alster.
\end{thm}
\begin{proof}
Since first countable $T_2$ Lindel\"of spaces have cardinality no more than that of the continuum, it suffices to show $X$ is Lindel\"of. Take a countably closed elementary submodel $M$ of some sufficiently large $H_\theta$ which contains $X$ and its topology. Then $X \cap M$ ($= X_M$) is a closed subspace of $X$ \cite{Junqueira1998}. Therefore $X \cap M$ is Alster and hence Lindel\"of. But then by \cite{Junqueira1998}, $X$ is Lindel\"of.
\end{proof}

Similar arguments clearly work if we replace ``Alster'' by ``$\sigma$-compact'' or other strengthenings of ``Lindel\"of'' in the statement of the Theorem.

\begin{cor}
$CH$ implies that if $X$ is first countable $T_3$ and each closed subspace of size $\leq \aleph_1$ is productively $FC$-Lindel\"of, then $X$ is Alster.
\end{cor}

\begin{cor}
$CH$ implies that if $X$ is metrizable and each closed subspace of $X$ of size $\leq \aleph_1$ is productively $FC$-Lindel\"of, then $X$ is $\sigma$-compact.
\end{cor}

\begin{proof}
In a first countable space, every subspace of size $\leq 2^{\aleph_0}$ has weight $\leq 2^{\aleph_0}$.
\end{proof}

\section{Other forcings}

We can also consider preservation of Lindel\"of and $D$ by other kinds of forcing.  For example, it is known that a space is Lindel\"of in a Cohen or random real extension if and only if it is in the ground model~\cite{Dow1990}, \cite{Grunberg1998}, \cite{Scheepers}, and \cite{TallProblems}.  The situation for $D$ is more complicated; in \cite{AurJunLar} it is shown that a Lindel\"of space $X$ becomes a $D$-space in an extension by more than $|X|$ Cohen reals.  It follows immediately from Lemma \ref{lem2} and \cite{Scheepers} that this can be improved to:

\begin{thm}
Adding $\aleph_1$ Cohen reals makes a Lindel\"of space $D$.
\end{thm}
\begin{proof}
By \cite{Scheepers} that makes the space Rothberger and hence Menger; by Lemma \ref{lem2} it is hence $D$.
\end{proof}

%

We do not know the answer to the following:

\begin{pbm}[ 6]
Suppose $X$ is Lindel\"of in the ground model and $D$ in a random real extension.  Must $X$ be $D$ in the ground model?
\end{pbm}

We conjecture ``yes'', at least for 0-dimensional $X$.  In~\cite{Scheepers}, it is shown that if a space is Menger in a random extension, then it is Menger in the ground model.

\section{Examples and implications}

Figure \ref{thediagram} illustrates the relationships among the properties we have discussed. For convenience, we assume $T_3$ throughout the diagram and examples.  Zdomskyy~\cite{Zdomskyy2005} proved that if $\mathfrak{u} < \mathfrak{g}$, then Rothberger spaces are Hurewicz.  An easier proof is in \cite{Tsaban2008}.  Tall \cite{TallNote} proved this from Borel's Conjecture. Moore's $L$-space \cite{Moore2006} is Rothberger and Hurewicz \cite{Scheepers}, but is not productively Lindel\"of.  To see this, we note that Tsaban and Zdomskyy \cite{TZ} have shown that there is an $n \in \omega$ such that $L^n$ is not Lindel\"{o}f, so $L$ is not powerfully Lindel\"{o}f.  Let $n_0$ be the least such $n$.  Then $L^{n_0 - 1}$ is Lindel\"{o}f, but $L \times L^{n_0 - 1}$ is not.  We do not know whether, as claimed in \cite{Scheepers}, $L^2$ is not Lindel\"{o}f.  In an earlier version of this note, we asked whether indestructibly productively Lindelof spaces are D. In fact, they are Hurewicz \cite{TallNote}.  Also in \cite{TallNote} we show that indestructibly productively Lindel\"of spaces are powerfully Lindel\"of.  A number next to a solid arrow means that the example with that number from the list below shows that the arrow does not reverse.  A number next to a broken dashed arrow means that that example shows that the implication does not hold.  A dotted arrow indicates that the implication holds under the indicated hypothesis.

The numbers refer to the following examples:
\begin{enumerate}
\item Moore's $L$-space \cite{Moore2006}.
\item $[0,1]$.
\item The space $\mb{P}$ of irrationals.
\item $2^{\omega_1}$.
\item A Hurewicz (and hence Menger) subspace of the real line which is not $\sigma$-compact (and hence not Alster) \cite{Just1996}, \cite{Tsaban2008}, \cite{Ts}.
\item Michael's space \cite{Michael1971}.
\item The one-point Lindel\"ofication of the discrete space of size $\aleph_1$.
\item A Bernstein set \cite{Michael1971}.
\item Przymusi\'nski's space \cite{Przymusinski1980}.
\item Another example in \cite{Przymusinski1980}.
\item The \textit{Sorgenfrey line} is well-known to be Lindel\"of, have closed sets $G_\delta$, and to have non-Lindel\"of square; on the other hand, the product of a Lindel\"of space with closed sets $G_\delta$ with a separable Lindel\"of space (such as $\mb{P}$) is Lindel\"of \cite{Przymusinski1984}.
\item A Menger subspace of the real line which is not Hurewicz \cite{ChaberPol}, \cite{Tsaban2008}, \cite{Ts}.
\item The subspace of the Michael line obtained from a set concentrated on the rationals.
\end{enumerate}

\noindent Some of the most interesting problems from the diagram are:
\begin{enumerate}[A.]
\item Is there a ZFC example of a Lindel\"of space whose product with $\mb{P}$ is not Lindel\"of?
\item Is there a productively Lindel\"of space which is not powerfully Lindel\"of?
\item Is there a productively Lindel\"of space which is not Alster?
\end{enumerate}

\begin{landscape}
\begin{figure*}[h!]
\centering

\psmatrix[mcol=c,nodealign=true]
            & Lindel\"of $P$ & & $\sigma$-compact\\
\hspace{-50pt}Rothberger  &                           & \begin{tabular}{c}indestructibly\\ productively\\ Lindel\"of\end{tabular}\\[30pt]
            & \begin{tabular}{c}indestructibly\\ Lindel\"of\end{tabular} & Hurewicz & Alster\\[36pt]
            & Menger                    & & \begin{tabular}{c}productively\\ Lindel\"of \end{tabular} & \begin{tabular}{c}productively\\$FC$-Lindel\"of\end{tabular}&\begin{tabular}{c}powerfully\\ Lindel\"of\end{tabular}\\
            & Lindel\"of $D$            & Lindel\"of & $\times\mathbb{P}$ Lindel\"of & $X^2$ Lindel\"of & \begin{tabular}{c}finitely\\ powerfully\\ Lindel\"of\end{tabular}
\endpsmatrix
\psset{arrows=->,labelsep=3pt,nodesep=3pt}

\ncarc[arcangle=35,linestyle=dashed]{1,2}{1,4}\naput{7}\ncput[nrot=:90,npos=0.6]{\psline[xunit=0.25cm,arrows=-](-1,0)(1,0)}
\ncline{1,2}{2,1}\naput{1}
\ncline{1,2}{3,2}\trput[tpos=0.4]{3}\tlput[tpos=0.4]{BC}
\ncline{1,2}{2,3}\naput{2}
\ncarc[arcangle=-22,linestyle=dashed]{1,4}{1,2}\naput{2}\ncput[nrot=:90,npos=0.6]{\psline[xunit=0.25cm,arrows=-](-1,0)(1,0)}
\ncarc[arcangle=-12]{1,4}{3,4}\tlput{7}
\ncangle[arrows=<->,linestyle=dotted,angleA=180,angleB=90,linearc=1]{1,4}{2,3}\nbput[npos=0.4]{\small metrizable}
\ncarc[linestyle=dashed]{1,4}{2,3}\naput{4}\ncput[nrot=:90,npos=0.6]{\psline[xunit=0.25cm,arrows=-](-1,0)(1,0)}
\ncline{1,4}{3,3}\nbput[npos=0.4]{5}
\ncline{2,3}{4,6}\naput[npos=0.8]{4}
\ncline{2,3}{3,3}\naput[npos=0.3]{1}
\ncline{2,1}{3,2}\nbput[npos=0.7]{3}
\ncline[linestyle=dotted,linewidth=1pt]{2,1}{3,3}\nbput[npos=0.3,nrot=:0]{\small $\mathfrak{u}<\mathfrak{g}$}
\ncline[linestyle=dashed]{2,1}{3,4}\trput{1}\ncput[nrot=:90,npos=0.6]{\psline[xunit=0.25cm,arrows=-](-1,0)(1,0)}
\ncline{2,1}{4,2}\tlput{2}
\ncline{2,3}{4,4}\naput[npos=0.27]{4}
\ncangle[linestyle=dashed,linearc=0.6,angleA=270,angleB=270,arm=0.7cm]{2,1}{5,4}\nbput[npos=0.57]{13 $(CH)$}\ncput[nrot=:90,npos=0.6]{\psline[xunit=0.25cm,arrows=-](-1,0)(1,0)}

\ncangles[linestyle=dashed,linearc=0.6,angleA=90,angleB=0,armA=5cm, armB=1cm]{2,1}{5,6}\nbput[npos=0.47]{1}
\ncput[nrot=:90,npos=0.6]{\psline[xunit=0.25cm,arrows=-](-1,0)(1,0)}

\ncarc[linestyle=dashed,arcangle=12]{2,3}{1,4}\naput{7}\ncput[nrot=:90,npos=0.6]{\psline[xunit=0.25cm,arrows=-](-1,0)(1,0)}
\ncline{2,3}{3,2}\nbput[npos=0.2]{8}
\ncline[linestyle=dashed]{3,2}{3,3}\nbput[npos=0.6]{4}\ncput[nrot=:90,npos=0.4]{\psline[xunit=0.25cm,arrows=-](-1,0)(1,0)}
\ncarc[arcangle=12,linestyle=dashed]{3,2}{4,2}\tlput{4~~~}\ncput[nrot=:90,npos=0.3]{\psline[xunit=0.25cm,arrows=-](-1,0)(1,0)}
\ncarc[arcangle=12,linestyle=dashed]{3,2}{4,4}\taput{8}\ncput[nrot=:90]{\psline[xunit=0.25cm,arrows=-](-1,0)(1,0)}
\ncline{3,3}{4,2}\tlput{12}
\ncline[linestyle=dashed]{3,3}{4,4}\trput{1}\ncput[nrot=:90,npos=0.4]{\psline[xunit=0.25cm,arrows=-](-1,0)(1,0)}
\ncarc[arcangle=-12,linestyle=dotted,linewidth=1pt]{3,4}{1,4}\trput{\begin{tabular}{l}\small compact\\\small sets $G_\delta$\end{tabular}}
\ncarc[arcangle=-12,linestyle=dashed]{3,4}{3,2}\nbput[npos=0.68]{4}\ncput[nrot=:90,npos=0.6]{\psline[xunit=0.25cm,arrows=-](-1,0)(1,0)}
\ncline{3,4}{3,3}\naput[npos=0.75]{5}
\ncline{3,4}{4,2}\naput[npos=0.8]{5}
\ncarc[arcangle=-12]{3,4}{4,4}
\ncangle[angleB=90,linearc=1]{3,4}{4,6}\taput{3}
\ncarc[arcangle=12,linestyle=dashed]{4,2}{3,2}\trput{~~~3}\ncput[nrot=:90,npos=0.4]{\psline[xunit=0.25cm,arrows=-](-1,0)(1,0)}
\ncline[linestyle=dashed]{4,2}{4,4}\nbput{1}\ncput[nrot=:90,npos=0.4]{\psline[xunit=0.25cm,arrows=-](-1,0)(1,0)}
\ncline{4,2}{5,2}\tlput{3}
\ncline{4,2}{5,3}\naput{3}
\ncarc[arcangle=12,linestyle=dashed]{4,4}{3,2}\tbput{4}\ncput[nrot=:90]{\psline[xunit=0.25cm,arrows=-](-1,0)(1,0)}
\ncarc[arcangle=-12,linestyle=dotted,linewidth=1pt]{4,4}{3,4}\trput{\small~$w\leq\aleph_1~(CH)$}
\ncline{4,4}{4,5}
\ncline{4,4}{5,3}\nbput{1}
\ncline{4,4}{5,4}\nbput{11}
\ncline{4,4}{5,6}
\ncline{4,5}{5,4}\nbput{11}
\ncarc[arcangle=-20,linestyle=dashed]{4,6}{4,4}\taput{8}\ncput[nrot=:90,npos=0.4]{\psline[xunit=0.25cm,arrows=-](-1,0)(1,0)}
\ncline{4,6}{5,4}\nbput[npos=0.3]{11}
\ncline{4,6}{5,5}\naput[npos=0.4]{9}
\ncline{4,6}{5,6}\trput{9}
\ncline{5,2}{5,3}
\ncarc[arcangle=-12,linestyle=dashed]{5,3}{5,4}\tbput{$6~(CH)$}\ncput[nrot=:90]{\psline[xunit=0.25cm,arrows=-](-1,0)(1,0)}
\ncarc[arcangle=-12]{5,4}{5,3}
\ncline[linestyle=dashed]{5,4}{5,5}\naput[npos=0.4]{11}\ncput[nrot=:90,npos=0.6]{\psline[xunit=0.25cm,arrows=-](-1,0)(1,0)}
\ncarc[linestyle=dashed,arcangle=12]{5,6}{5,4}\naput{13 $(CH)$}\ncput[nrot=:90,npos=0.35]{\psline[xunit=0.25cm,arrows=-](-1,0)(1,0)}
\ncline{5,6}{5,5}\taput{10}

\caption{The relationships among various properties discussed.}\label{thediagram}
\end{figure*}
\end{landscape}

\begin{rem}
Since this paper was submitted, there have been several developments worth noting:
\begin{enumerate}
\item{
There is an easy proof that CH implies productively Lindel\"{o}f spaces are Menger \cite{T}.
}
\item{
The completeness requirement in d) $\Rightarrow$ a) of Theorem \ref{thm14} has been removed \cite{RZ}.
}
\item{
Further investigation of the influence of small cardinals on Michael's problem can be found in \cite{AAJT}.
}
\end{enumerate}
\end{rem}

In conclusion, we thank the careful referee and Boaz Tsaban for many helpful suggestions and
for catching several errors.

\nocite{*}
\bibliographystyle{acm}
\bibliography{lsipd}

\newcommand{\noopsort}[1]{}
\begin{thebibliography}{10}

\bibitem{AAJT}
{\sc Alas, O.~T., Aurichi, L.~F., Junqueira, L.~R., and Tall, F.~D.}
\newblock Non-productively {L}indel\"of spaces and small cardinals.
\newblock {\em Houston J. Math.\/}.
\newblock In press.

\bibitem{Alster1987}
{\sc Alster, K.}
\newblock On spaces whose product with every {L}indel\"of space is
  {L}indel\"of.
\newblock {\em Coll. Math. 54\/} (1987), 171--178.

\bibitem{Alster1988}
{\sc Alster, K.}
\newblock On the class of all spaces of weight not greater than $\omega_1$
  whose {C}artesian product with every {L}indel\"of space is {L}indel\"of.
\newblock {\em Fund. Math. 129\/} (1988), 133--140.

\bibitem{Aurichi}
{\sc Aurichi, L.~F.}
\newblock {$D$}-spaces, topological games and selection principles.
\newblock {\em Topology Proc. 36\/} (2010), 107--122.

\bibitem{AurJunLar}
{\sc Aurichi, L.~F., Junqueira, L.~R., and Larson, P.~B.}
\newblock {$D$}-spaces, irreducibility, and trees.
\newblock {\em Topology Proc. 35\/} (2010), 73--82.

\bibitem{BZ}
{\sc Banakh, T., and Zdomskyy, L.}
\newblock Separation properties between the $\sigma$-compactness and {H}urewicz
  property.
\newblock {\em Topology Appl. 156\/} (2008), 10--15.

\bibitem{Barr2000}
{\sc Barr, M., Kennison, J.~F., and Raphael, R.}
\newblock On productively {L}indel\"of spaces.
\newblock {\em Sci. Math. Jpn. 65\/} (2000), 319--332.

\bibitem{Borges1991}
{\sc Borges, C.~R., and Wehrly, A.~C.}
\newblock A study of {$D$}-spaces.
\newblock {\em Topology Proc. 16\/} (1991), 7--15.

\bibitem{ChaberPol}
{\sc Chaber, J., and Pol, R.}
\newblock A remark on {F}remlin-{M}iller theorem concerning the {M}enger
  property and {M}ichael concentrated sets.
\newblock Unpublished note.

\bibitem{vanDouwen1984}
{\sc {\noopsort{Douwen}}van~Douwen, E.~K.}
\newblock The integers and topology.
\newblock In {\em Handbook of {S}et-theoretic {T}opology}, K.~Kunen and J.~E.
  Vaughan, Eds. North-Holland, Amsterdam, 1984, pp.~111--167.

\bibitem{vanDouwen1979}
{\sc {\noopsort{Douwen}}van~Douwen, E.~K., and Pfeffer, W.~F.}
\newblock Some properties of the {S}orgenfrey line and related spaces.
\newblock {\em Pacific J. Math. 81\/} (1979), 371--377.

\bibitem{Dow1990}
{\sc Dow, A., Tall, F.~D., and Weiss, W. A.~R.}
\newblock New proofs of the consistency of the normal {M}oore space conjecture,
  {II}.
\newblock {\em Topology Appl. 37}, 2 (1990), 115--129.

\bibitem{Eisworth2007}
{\sc Eisworth, T.}
\newblock On {$D$}-spaces.
\newblock In {\em Open {P}roblems in {T}opology {II}}, E.~Pearl, Ed. Elsevier,
  Amsterdam, 2007, pp.~129--134.

\bibitem{Engelking1989}
{\sc Engelking, R.}
\newblock {\em General {T}opology}.
\newblock Heldermann Verlag, Berlin, 1989.

\bibitem{Gerlits1982}
{\sc Gerlits, J., and Nagy, Z.}
\newblock Some properties of ${C(X)}$, {I}.
\newblock {\em Topology Appl. 14\/} (1982), 151--161.

\bibitem{Gruenhage2009}
{\sc Gruenhage, G.}
\newblock A survey of {$D$}-spaces.
\newblock In \emph{{S}et {T}heory and its {A}pplications}, {C}ontemp. {M}ath.,
  ed. {L}. {B}alinkostova, {A}. {C}aicedo, {S}. {G}eschke, {M}. {S}cheepers,
  2011, pp. 13--28.

\bibitem{Grunberg1998}
{\sc Grunberg, R., Junqueira, L.~R., and Tall, F.~D.}
\newblock Forcing and normality.
\newblock {\em Topology Appl. 84\/} (1998), 145--174.

\bibitem{Hurewicz1925}
{\sc Hurewicz, W.}
\newblock Uber eine {V}erallgemeinerung des {B}orelschen {T}heorems.
\newblock {\em Math. Zeit. 24\/} (1925), 401--421.

\bibitem{Iwasa2007}
{\sc Iwasa, A.}
\newblock Covering properties and {C}ohen forcing.
\newblock {\em Topology. Proc. 31\/} (2007), 553--559.

\bibitem{Juhasz1989}
{\sc Juh\'asz, I., and Weiss, W.}
\newblock Omitting the cardinality of the continuum in scattered spaces.
\newblock {\em Topology Appl. 31\/} (1989), 19--27.

\bibitem{Junqueira2000}
{\sc Junqueira, L.~R.}
\newblock Upwards preservation by elementary submodels.
\newblock {\em Topology Proc. 25\/} (Spring 2000), 225--249.

\bibitem{Junqueira1998}
{\sc Junqueira, L.~R., and Tall, F.~D.}
\newblock The topology of elementary submodels.
\newblock {\em Topology Appl. 82\/} (1998), 239--266.

\bibitem{Just1996}
{\sc Just, W., Miller, A.~W., Scheepers, M., and Szeptycki, P.~J.}
\newblock Combinatorics of open covers {(II)}.
\newblock {\em Topology Appl. 73\/} (1996), 241--266.

\bibitem{Kechris1994}
{\sc Kechris, A.~S.}
\newblock {\em Classical {D}escriptive {S}et {T}heory}.
\newblock Springer-Verlag, New York, 1994.

\bibitem{Michael1963}
{\sc Michael, E.~A.}
\newblock The product of a normal space and a metric space need not be normal.
\newblock {\em Bull. Amer. Math. Soc. 69\/} (1963), 376.

\bibitem{Michael1971}
{\sc Michael, E.~A.}
\newblock Paracompactness and the {L}indel\"of property in finite and countable
  {C}artesian products.
\newblock {\em Compositio Math. 23\/} (1971), 199--214.

\bibitem{Miller1984}
{\sc Miller, A.~W.}
\newblock Special subsets of the real line.
\newblock In {\em Handbook of {S}et-{T}heoretic {T}opology}, K.~Kunen and J.~E.
  Vaughan, Eds. North-Holland, Amsterdam, 1984, pp.~685--732.

\bibitem{Moore1999}
{\sc Moore, J.~T.}
\newblock Some of the combinatorics related to {M}ichael's problem.
\newblock {\em Proc. Amer. Math. Soc. 127\/} (1999), 2459--2467.

\bibitem{Moore2006}
{\sc Moore, J.~T.}
\newblock A solution to the {$L$}-space problem.
\newblock {\em J. Amer. Math. Soc. 19\/} (2006), 717--736.

\bibitem{Przymusinski1980}
{\sc Przymusi\'nski, T.~C.}
\newblock Normality and paracompactness in finite and countable {C}artesian
  products.
\newblock {\em Fund. Math. 105\/} (1980), 87--104.

\bibitem{Przymusinski1984}
{\sc Przymusi\'nski, T.~C.}
\newblock Products of normal spaces.
\newblock In {\em Handbook of {S}et-{T}heoretic {T}opology}, K.~Kunen and J.~E.
  Vaughan, Eds. North-Holland, Amsterdam, 1984, pp.~781--826.

\bibitem{R}
{\sc Rec\l{}aw, I.}
\newblock Every {L}usin set is undetermined in the point-open game.
\newblock {\em Fund. Math. 144\/} (1994), 43--54.

\bibitem{RZ}
{\sc Repov\v{s}, D., and Zdomskyy, L.}
\newblock On the {M}enger covering property and {$D$} spaces.
\newblock {P}roc. {A}mer. {M}ath. {S}oc. {T}o appear.

\bibitem{Rudin1975}
{\sc Rudin, M.~E., and Starbird, M.}
\newblock Products with a metric factor.
\newblock {\em Topology Appl. 5\/} (1975), 235--248.

\bibitem{Scheepers}
{\sc Scheepers, M., and Tall, F.~D.}
\newblock {L}indel\"of indestructibility,\newline topological games and
  selection principles.
\newblock {\em Fund. Math. 210\/} (2010), 1--46.

\bibitem{T}
{\sc Tall, F.~D.}
\newblock Productively {L}indel\"of spaces may all be {$D$}.
\newblock Canad. Math. Bull. {T}o appear.

\bibitem{TallNote}
{\sc {\noopsort{TallA}}Tall, F.~D., and Tsaban, B.}
\newblock On productively {L}indel\"of spaces.
\newblock Topology Appl., to appear.

\bibitem{Tall}
{\sc {\noopsort{TallB}}Tall, F.~D.}
\newblock {L}indel\"of spaces which are {M}enger, {H}urewicz, {A}lster,
  productive, or {$D$}.
\newblock {\em Topology Appl.\/}.
\newblock To appear.

\bibitem{TallProblems}
{\sc {\noopsort{TallC}}Tall, F.~D.}
\newblock Some problems and techniques in set-theoretic topology.
\newblock In \emph{{S}et {T}heory and its {A}pplications}, {C}ontemp. {M}ath.,
  ed. {L}. {B}alinkostova, {A}. {C}aicedo, {S}. {G}eschke, {M}. {S}cheepers,
  2011, pp. 183--209.

\bibitem{Tall1995}
{\sc {\noopsort{TallD}}Tall, F.~D.}
\newblock On the cardinality of {L}indel\"of spaces with points {$G\sb
  \delta$}.
\newblock {\em Topology Appl. 63}, 1 (1995), 21--38.

\bibitem{Tall1976}
{\sc {\noopsort{TallE}}Tall, F.~D.}
\newblock The density topology.
\newblock {\em Pacific J. Math. 62\/} (1976), 175--184.

\bibitem{Ts}
{\sc Tsaban, B.}
\newblock {M}enger's and {H}urewicz's {P}roblems: {S}olutions from ``{T}he
  {B}ook" and refinements.
\newblock In \emph{{S}et {T}heory and its {A}pplications}, {C}ontemp. {M}ath.,
  ed. {L}. {B}alinkostova, {A}. {C}aicedo, {S}. {G}eschke, {M}. {S}cheepers,
  2011, pp. 211--226.

\bibitem{TZ}
{\sc Tsaban, B., and Zdomskyy, L.}
\newblock Arhangel'ski\u{\i} sheaf amalgamation in topological groups.
\newblock In preparation.

\bibitem{TsabanZdomskyy2008}
{\sc Tsaban, B., and Zdomskyy, L.}
\newblock Combinatorial images of sets of reals and semifilter trichotomy.
\newblock {\em J. Symbolic Logic 73\/} (2008), 1278--1288.

\bibitem{Tsaban2008}
{\sc Tsaban, B., and Zdomskyy, L.}
\newblock Scales, fields, and a problem of {H}urewicz.
\newblock {\em J. European Math. Soc. 10\/} (2008), 837--866.

\bibitem{Zdomskyy2005}
{\sc Zdomskyy, L.}
\newblock A semifilter approach to selection principles.
\newblock {\em Comment. Math. Univ. Carolinae 46\/} (2005), 525--539.

\end{thebibliography}

\noindent
{\rm Leandro F. Aurichi\\
Instituto de Ci\^{e}ncas Matem\'{a}ticas de Computa\c{c}\~{a}o
(ICMC-USP), Universidade De
S\~{a}o Paulo, S\~{a}o Carlos, SP\\
CEP 13566-590 - Brazil\\}
\noindent
{\it e-mail address:} {\rm aurichi@icmc.usp.br}\\
\\
\noindent
{\rm Franklin D. Tall\\
Department of Mathematics\\
University of Toronto\\
Toronto, Ontario\\
M5S 2E4\\
CANADA\\}
\noindent
{\it e-mail address:} {\rm f.tall@utoronto.ca}

\end{document}